\documentclass[12 pt]{amsart}

\usepackage[toc,page]{appendix}
\usepackage{etex}
\usepackage[shortlabels]{enumitem}
\usepackage{amsmath}
\usepackage{amsxtra}
\usepackage{amscd}
\usepackage{amsthm}
\usepackage{amsfonts}
\usepackage{amssymb}
\usepackage{eucal}
\usepackage[all]{xy}
\usepackage{graphicx}
\usepackage{tikz-cd}
\usepackage{mathrsfs}
\usepackage{mathpazo}
\usepackage{euler}
\usepackage[colorinlistoftodos, textsize=tiny]{todonotes}
\usepackage{morefloats}
\usepackage{mathdots}
\usepackage{pdfpages}
\usepackage{thm-restate}
\usepackage[utf8]{inputenc}
\usepackage{epigraph}
\usepackage{csquotes}
\usepackage{url}

\graphicspath{ {images/} }

\RequirePackage{color}
\definecolor{myred}{rgb}{0.75,0,0}
\definecolor{mygreen}{rgb}{0,0.5,0}
\definecolor{myblue}{rgb}{0,0,0.65}

\usepackage{hyperref}
\hypersetup{citecolor=blue}
\usepackage{tikz}
\usetikzlibrary{matrix,arrows,decorations.pathmorphing}

\theoremstyle{plain}
\newtheorem{theorem}{Theorem}[section]
\newtheorem{proposition}[theorem]{Proposition}
\newtheorem{lemma}[theorem]{Lemma}
\newtheorem{corollary}[theorem]{Corollary}
\theoremstyle{definition}

\newtheorem{remark}[theorem]{Remark}

\theoremstyle{remark}

\numberwithin{equation}{section}
  
\newcommand\nc{\newcommand}
\nc\on{\operatorname}
\nc\renc{\renewcommand}

\newcommand\bp{{\mathbb P}}

\newcommand\sce{\mathscr E}

\newcommand\sch{\mathscr H}

\newcommand\sco{\mathscr O}

\newcommand\scv{\mathscr V}

\newcommand\scx{\mathscr X}
\newcommand\scy{\mathscr Y}
\newcommand\scz{\mathscr Z}

\newcommand \ra{\rightarrow}

\newcommand \minhilb[2]{\sch^{\text{scroll}}_{#1, #2}}
\newcommand \hilb[1]{\sch_{#1}}
\newcommand \uhilb[1]{\scv_{#1}}

\newcommand \scroll[1]{\text{S}_{#1}}
\newcommand \minhilbsmooth[2]{\sch^{\text{scroll}}_{#1, #2, \text{sm}}}

\newcommand \broken[2]{\sch^{broken}_{#1, #2}}

\newcommand \gbundle [2]{\mathbb G (#1, #2)}

\makeatletter
\newcommand{\customlabel}[2]{%
   \protected@write \@auxout {}{\string \newlabel {#1}{{#2}{\thepage}{#2}{#1}{}} }%
   \hypertarget{#1}{#2}
}

\DeclareMathOperator\spn{Span}

\newcommand\bk{{\Bbbk}}

\def\listtodoname{List of Todos}
\def\listoftodos{\@starttoc{tdo}\listtodoname}

\title{Interpolation of Varieties of Minimal Degree}
\author{Aaron Landesman} \address{Dept. of Mathematics, Stanford University,
	Stanford, CA 94305-2125} 
\email{aaronlandesman@stanford.edu}
\date{\today}

\begin{document}
\maketitle

\begin{abstract}
It is well known that one can find a rational normal curve in $\mathbb P^n$ through $n+3$ general points.  We prove a generalization of this to higher dimensional varieties, showing that smooth varieties of minimal degree can be interpolated through points and linear spaces.
\end{abstract}

\section{Introduction}

In this paper, we investigate the question of interpolation,
which asks whether
we can fit a given type of variety through a general collection
of points in projective space.
As one of the simplest examples of interpolation, if one chooses
$5$ points in $\bp^2$, one can always find a degree $2$ curve
passing through those five points, so degree 2 curves in $\bp^2$
satisfy interpolation.

Recall that a $k$ dimensional variety of degree $d$ in $\bp^n$ is of {\bf minimal degree}
if it is nondegenerate and $d = n + 1 - k$.
There is a simple classification of varieties of minimal degree
as detailed in \cite[Theorem 1]{eisenbudH:on-varieties-of-minimal-degree}:
If $X$ is a smooth irreducible nondegenerate variety of minimal degree in $\bp^n$,
	then $X$ is either a quadric hypersurface,
	a rational normal scroll, the image of
	$\bp^2$ under the 2-Veronese embedding
	$\nu_2(\bp^2) \rightarrow \bp^5$,
	or $\bp^n$ itself.
In this paper, we show that smooth varieties of minimal degree satisfy 
interpolation.

We next define more precisely what we mean by interpolation.
Suppose $U$ is an integral subscheme of the Hilbert scheme
parameterizing varieties of dimension $k$ in $\bp^n$.
Define $q$ and $r$ so that $\dim U = q \cdot (n- k) + r$.
We say $U$ {\bf satisfies interpolation} if for any collection
$p_1, \ldots, p_q, \Lambda$, where $p_i \in \bp^n$ are points
and $\Lambda \subset \bp^n$ is a plane of dimension
$n-k-r$, there exists some $[Y] \in U$ so that
$Y$ passes through $p_1, \ldots, p_q$ and meets $\Lambda$.
We say $U$ {\bf satisfies weak interpolation} if there exists
some $[Y] \in U$ meeting $q$ general points $p_1, \ldots, p_q$.
If $X$ is a projective variety lying on a unique
irreducible component of the Hilbert scheme, denoted
$\hilb X$, then we say $X$ satisfies interpolation if $\hilb X$ does.
Although this description of interpolation,
given in \cite[Theorem A.7(9)]{landesmanP:interpolation-problems:-del-pezzo-surfaces}, is the most classical one,
there are at least twenty two equivalent descriptions of interpolation
under moderate hypotheses, as described in
\cite[Theorem A.7]{landesmanP:interpolation-problems:-del-pezzo-surfaces}.

Interpolation has applications to Gromov-Witten theory,
the slope conjecture,
constructing degenerations,
and understanding the Hilbert scheme, as detailed in
\cite[Subsection 1.1]{landesmanP:interpolation-problems:-del-pezzo-surfaces}. Also see \cite[Section 1]{landesmanP:interpolation-problems:-del-pezzo-surfaces} for a more leisurely introduction to interpolation.
It is known that nearly all nonspecial curves satisfy interpolation
\cite{atanasovLY:interpolation-for-normal-bundles-of-general-curves},
and that canonical curves in genus other than $6$ and $4$ satisfy
weak interpolation as proven in
\cite[Chapter 13]{stevens:deformations-of-singularities}.
Most recently, it was shown that smooth del Pezzo surfaces satisfy
weak interpolation \cite[Theorem 1.1]{landesmanP:interpolation-problems:-del-pezzo-surfaces}. 

We now state our main result, which continues
the investigation of interpolation of higher dimensional varieties.
The theorem holds over an algebraically closed field of
arbitrary characteristic.

\begin{restatable}{theorem}{minimalInterpolation}
	\label{theorem:interpolation-minimal-surfaces}
	Smooth varieties of minimal degree satisfy interpolation.
\end{restatable}

As a consequence, varieties of minimal degree satisfy strong interpolation
(which is a stronger notion of interpolation discussed preceding \autoref{corollary:strong-interpolation-for-varieties-of-minimal-degree})
in characteristic $0$, as proven in \autoref{corollary:strong-interpolation-for-varieties-of-minimal-degree}.

\begin{remark}
	\label{remark:}
Although several bits and pieces of
\autoref{theorem:interpolation-minimal-surfaces} were previously known, the unknown cases
are the trickiest ones to deal with in the proof we present.

We now describe those parts of \autoref{theorem:interpolation-minimal-surfaces} which were previously
established. 
The dimension $1$ case is the well known fact that one can find
	a rational normal curve through $n+3$ points in $\bp^n$.
	The Veronese surface was shown to satisfy interpolation
	in \cite[Theorem 19]{coble:associated-sets-of-points},
	which is summarized in a more modern treatment in
	\cite[Theorem 5.6]{landesmanP:interpolation-problems:-del-pezzo-surfaces},
	piecing together the ideas in 
	\cite[Theorem 5.2, Theorem 5.6]{dolgachev:on-certain-families-of-elliptic-curves-in-projective-space}.
	It was already established that $2$-dimensional
	scrolls satisfy interpolation in
	Coskun's thesis ~\cite[Example, p.\ 2]{coskun:degenerations-of-surface-scrolls}, and furthermore,
	Coskun gives a method for computing the number of scrolls
	meeting a specified collection of general linear spaces.
	Finally, weak interpolation was established for
	scrolls of degree $d$ and dimension $k$ with
	$d \geq 2k - 1$ via the Gale transform in \cite[Theorem 4.5]{eisenbudP:the-projective-geometry-of-the-gale-transform}
	and independently via an explicit construction in \cite[Theorem 3.2]{cavaliere1995quadrics}.
	Further, in the case $k \leq d < 2k - 1$, it was shown in \cite[Theorem 3.2]{cavaliere1995quadrics}
	that scrolls pass through many points, though the resulting number is one point
	short of establishing weak interpolation.

While our methods are similar to those of \cite{coskun:degenerations-of-surface-scrolls},
they differ drastically from those of 
\cite[Theorem 4.5]{eisenbudP:the-projective-geometry-of-the-gale-transform}
and \cite[Theorem 3.2]{cavaliere1995quadrics}.
\end{remark}

\subsection{Notations and Conventions}
\label{subsection:notation}

We work over an algebraically closed field $\bk$ of arbitrary
characteristic.
We notate a sequence $(a, \ldots, a)$ of length $b$ as $(a^b)$. So, for example,
we would notate $(1,1,1,2,3,3)$ as $(1^3, 2, 3^2)$.
We also include the following idiosyncratic notation, mostly
regarding scrolls:
\begin{enumerate}
	\item When $X$ lies on a unique irreducible
		component of the Hilbert scheme, we let
		$\hilb X$ denote that irreducible component
		and let $\uhilb X$ denote the universal family
		over that component.
	\item Recall that a scroll is a projective bundle $\pi:X \ra \bp^1$
		embedded into projective space by $\sco_\pi(1)$.
If $X \cong \bp \sce$ with $\sce \cong \sco(a_1) \oplus \cdots \oplus (a_k)$, we say
$X$ is of type $(a_1, \ldots, a_k)$. We use $\scroll {a_1, \ldots, a_k}$ to refer to a smooth scroll of type $(a_1, \ldots, a_k)$.
	\item When dealing with a scroll $X$ in projective space, we use $d$ to refer to its degree, $k$ to refer to its dimension,
		and $n := d + k-1$ to refer to the dimension of the ambient projective space, $X \subset \bp^n$.
	\item If $X$ is a smooth scroll of degree $d$ and dimension $k$, we use $\minhilb d k := \hilb X$.
	\item Let $X$ be a smooth scroll of degree $d$ and dimension $k$. We use $\minhilbsmooth d k$ to refer to the complement in
		$\hilb X$ of the image under $\pi:\uhilb X \ra \hilb X$ of the singular locus of the map $\pi:\uhilb X \ra \hilb X$.
	\item We use $\broken d k$ to denote the closure of the locus of points in $\minhilb d k$ corresponding to
		varieties which are the union of a $k$-plane and a degree $d-1$, dimension $k$ scroll,
		intersecting in a $(k-1)$-plane which is a ruling plane of the degree $d-1$ scroll.
\end{enumerate}

\subsection{Acknowledgements}

I thank Joe Harris and Anand Patel for advising this project.
For helpful conversations and correspondence, I thank
Atanas Atanasov,
Francesco Cavazzani, 
Brian Conrad,
Izzet Coskun,
Anand Deopurkar,
Phillip Engel, 
Changho Han,
Joe Harris,
Brendan Hassett,
Allen Knutsen,
Eric Larson,
John Lesieutre,
Rahul Pandharipande,
Anand Patel,
Alex Perry,
Geoffrey Smith,
Hunter Spink,
Ravi Vakil,
Adrian Zahariuc,
Yifei Zhao,
Yihang Zhu,
and David Zureick-Brown.

\section{Background on scrolls}
\label{subsection:fano-schemes-of-scrolls}

In this section, we collect background results on
several descriptions of the structure
of scrolls, which we later use to show scrolls satisfy
interpolation.
In \autoref{lemma:segre-isomorphic-to-scroll} we show certain 
Segre embeddings are scrolls, 
in \autoref{lemma:linear-spaces-in-varieties-of-minimal-degree},
we describe Fano schemes of scrolls, in
\autoref{proposition:smooth-minimal-degree-hilbert-scheme},
we show smooth varieties of minimal degree correspond to
smooth points
of the Hilbert scheme,
in \autoref{lemma:coskun-degenerations-of-scrolls}
we describe possible degenerations of scrolls,
and in
\autoref{lemma:unique-segre-containing-linear-spaces},
we describe one method of constructing certain scrolls.

\begin{lemma}
	\label{lemma:segre-isomorphic-to-scroll}
	The scroll $\scroll {1^k}$ of dimension $k$ in $\bp^{2k-1}$ is, up to automorphism of $\bp^{2k-1}$,
	isomorphic to the Segre embedding $\bp^1 \times \bp^{k-1} \ra \bp^{2k-1}$.
\end{lemma}
\begin{proof}
	Observe that by definition we have $\scroll {1^k} \cong \bp \left(  \oplus_{i=1}^k \sco_{\bp^1}(1)\right).$
	So, abstractly, $\scroll {1^k}$ is the trivial bundle over $\bp^1$, isomorphic to $\bp^1 \times \bp^{k-1}$.
	Further, the invertible sheaf $\sco_{\pi}(1)$ indeed embeds $\scroll {1^k}$ via the Segre embedding $\bp^1 \times \bp^{k-1} \ra \bp^{2k-1}$.
\end{proof}

\begin{proposition}
	\label{lemma:linear-spaces-in-varieties-of-minimal-degree}
	Suppose $X \subset \bp^n$ is a smooth scroll of minimal degree $d$ and dimension $k$, where $\pi: X \cong \bp\sce \ra \bp^1$ is the projection, for $\sce$ a
	locally free sheaf
	on $\bp^1$. Further, let $X \cong \scroll {a_1, \ldots, a_s, 1^j}$ with $a_1 \geq \cdots \geq a_s > 1$.
	Suppose that $X \subset \bp^n$ is embedded so that it is the span of the planes joining rational normal curves
	$C_1, \ldots, C_s, L_1, \ldots, L_j$ where $L_1, \ldots, L_j$ are lines and $C_1, \ldots, C_s$ are rational normal curves
	of degree at least $2$. Let $t$ be an integer with $1 \leq t \leq k - 1$.
Then, the $t$-planes contained in $X$ are of one of the following two forms:
	\begin{enumerate}
		\item[\customlabel{custom:line-small}{1}] If $j > 1$, let $P \subset \bp V$ be written as $P = \bp W$ for a two dimensional
			subspace $W \subset V$, and let $L_1, \ldots, L_j$ be the projectivizations of planes $P_1, \ldots, P_j \subset V$.
			Then, $W \in \spn(P_1, \ldots, P_j)$.
			
		\item[\customlabel{custom:line-large}{2}] $P$ is contained in some $(k-1)$-plane which is the fiber of the projection map $\pi:X \ra \bp^1$.
	\end{enumerate}
	The Fano scheme of $t$-planes in $X$ is smooth. If $j \geq 1$ and $t= 1$, it has two irreducible components, corresponding to planes of type
	\autoref{custom:line-small} and \autoref{custom:line-large}. It has one irreducible component 
	corresponding to planes of type \autoref{custom:line-large} otherwise.
	The component of planes of type \autoref{custom:line-large} is isomorphic to
the Grassmannian bundle $\gbundle {t+1} {\sce}$ over $\bp^1$. If it exists, meaning $t = 1$ and $j \geq 1$,
the component of planes of type \autoref{custom:line-small} is isomorphic to $\bp^{j-1}$.
\end{proposition}

One fairly
routine proof follows by first describing the Fano scheme first set theoretically,
then scheme theoretically, and then verifying
cohomologically that planes are smooth points of the Fano scheme.
See \cite[Proposition 5.2.2]{landesman:undergraduate-thesis} for a detailed proof.

\begin{proposition}
	\label{proposition:smooth-minimal-degree-hilbert-scheme}
	If $X$ is a smooth variety of minimal degree
	then $H^1(X, N_{X/\bp^n}) = 0$.
	In particular, such a variety is a smooth
	point of the Hilbert scheme.
	Further, if $X$ is a smooth scroll of minimal degree $d$ and dimension $k$ in $\bp^n$ then $h^0(X, N_{X/\bp^n}) = (d+k)^2 - k^2 - 3$. In particular, $\dim \minhilb d k =  (d+k)^2 - k^2 - 3$.
\end{proposition}
See \cite[Proposition 5.3.2, Proposition 5.3.5 and Proposition 5.3.16]{landesman:undergraduate-thesis} for a more detailed proof than the sketch given below.
\begin{proof}
	First, by \cite[Theorem 1]{eisenbudH:on-varieties-of-minimal-degree}, we only need
	verify this when $X$ is a quadric surface, the $2$ Veronese,
	or a smooth scroll.
	The case of hypersurfaces follows since
	for $X \subset \bp^n$ a hypersurface,
$N_{X/\bp^n} \cong \sco(\deg X)$.
The case of the $2$-Veronese follows from a fairly straightforward
chase of long exact sequences using the
normal exact sequence and Euler exact sequence.
Finally, the case of scrolls is somewhat more involved,
but follows from chasing exact sequences using
the normal exact sequence, the relative tangent exact sequence (see
\cite[Theorem 4.5.13]{brandenburg:tensor-categorical-foundations-of-algebraic-geometry}),
the Euler exact sequence,
and
the Leray spectral sequence.
\end{proof}

\begin{remark}
	\label{remark:grassmannian-compactification}
	For the remainder of the paper, instead of working
	in the compactification of the locus of smooth scrolls
	given by the Hilbert scheme, we work in the compactification
	of smooth scrolls of dimension $k$ and degree $d$
	given by rational curves of degree $d$ in the Grassmannian
	$G(k, k + d)$. This is described in the case
	$k = 2$ in \cite[Section 3]{coskun:degenerations-of-surface-scrolls},
	and the generalization to $k > 2$ is completely analogous.
	
	Since the Hilbert scheme and the compactification of rational
	curves in this Grassmannian are birational, the number of
	scrolls through a general collection of planes of specified
	dimensions is the same for either compactification.
	So to prove interpolation holds, it suffices to show
	there is a point in this compactification of rational curves
	corresponding to a scroll through a general collection
	of linear spaces.
\end{remark}

The following is a direct generalization of
\cite[Proposition 4.1]{coskun:the-enumerative-geometry-of-del-pezzo-surfaces-via-degenerations},
and the proof is nearly the same, mutatis mutandis.

\begin{lemma}[Generalization of \protect{\cite[Proposition 4.1]{coskun:degenerations-of-surface-scrolls}}]
	\label{lemma:coskun-degenerations-of-scrolls}
	Suppose $\scx \ra B$ is a flat family with $B$ a curve
so that the general fiber is a smooth scroll, and that the special fiber is nondegenerate.
Then, the special fiber is a connected variety whose irreducible components are
scrolls.
\end{lemma}

\begin{lemma}
	\label{lemma:unique-segre-containing-linear-spaces}
	Fix $k$ lines $\ell_1, \ldots, \ell_k \subset \bp^{2k-1}$ and three $(k-1)$-planes $\Lambda_1, \Lambda_2,\Lambda_3 \subset \bp^{2k-1}$.
	Suppose that $\ell_i \cap \Lambda_j$ consists of a single point for all pairs $(i,j)$, and that no two of these intersection points
	are the same. Further, suppose the lines $\ell_1, \ldots, \ell_k$ span $\bp^{2k-1}$.
	Then, there is a unique scheme $X \subset \bp^{2k-1}$ so that $X \cong \bp^1 \times \bp^{k-1}$ is the Segre embedding and $X$
	contains $\ell_1, \ldots, \ell_k, \Lambda_1, \Lambda_2, \Lambda_3$.

	Further, the lines $\ell_1, \ldots, \ell_k$ appear as the images of lines of the form $\bp^1 \times \left\{ q_k \right\} \subset \bp^1 \times \bp^{k-1}$
	under the Segre embedding, with $q_i \in \bp^{k-1}$.
\end{lemma}
For a more detailed proof, see \cite[Lemma 7.9.1]{landesman:undergraduate-thesis}.
\begin{proof}
	This is a special case of the fact that a scroll
	is uniquely determined as the planes swept out
	by a collection of rational normal curves with compatible
	isomorphisms.
\end{proof}

\section{Interpolation of varieties of minimal degree}
\label{section:interpolation-of-varieties-of-minimal-degree}

In this section, we prove that any smooth variety of minimal degree satisfies interpolation.
We work not in the compactification of smooth scrolls given
by the Hilbert scheme, but rather that given by rational curves in
the Grassmannian,
as described in \autoref{remark:grassmannian-compactification}.
By \cite[Theorem 1]{eisenbudH:on-varieties-of-minimal-degree}, we can
show this separately in the cases of a Hilbert scheme whose general member is a quadric hypersurface,
the Veronese surface in $\bp^5$ and a smooth rational normal scroll.
Since quadrics satisfy interpolation and the 2-Veronese
satisfies interpolation by
\cite[Theorem 5.6]{landesmanP:interpolation-problems:-del-pezzo-surfaces}, we shall concentrate
on showing that rational normal scrolls satisfy interpolation.
The general idea of the proof is to fix the dimension of scrolls, and induct on
the degree.
The elements of this inductive argument
is summarized verbally in \autoref{theorem:scrolls-interpolation} and
pictorially in \autoref{figure:induction-schematic}.
Most of the remaining subsections each prove one of the Propositions
appearing in \autoref{figure:induction-schematic}.

\subsection{The set-up for scrolls}
In order to show scrolls satisfy interpolation,
we start by stating the precise number of points and linear spaces a scroll must
pass through to satisfy interpolation.

\begin{lemma}
	\label{lemma:interpolation-numerics}
	The condition of interpolation means that we can find a degree $d$ dimension $k$ scroll 
	\begin{itemize}
		\item containing $d + 2k + 1$ general points and meeting a general $d-2k + 1$-plane if $d \geq 2k - 1$,
		\item containing $d + 2k + 2$ general points and meeting a general $2(d-k)$-plane if $k \leq d \leq 2k - 2$.
	\end{itemize}
\end{lemma}
\begin{proof}
	By \autoref{proposition:smooth-minimal-degree-hilbert-scheme}, we know
\begin{align*}
		\dim \minhilb d k
		&= (d+k)^2 - k^2- 3 \\
		&= (d-1) \left( d+2k + 1 \right) + 2k - 2.
	\end{align*}
	Note that we always have $d \geq k$ and $2k - 2 \leq d-1$ when $d \geq 2k - 1$. Therefore, in the case $d > 2k - 1$,
	interpolation means the scroll passes through $d + 2k + 1$ points and meets a $(d - 2k + 1)$-plane.
		Next, when $k < d \leq 2k - 2$, we have $d-1 <2k - 2 < 2(d-1)$, and so for interpolation to hold, such
	schemes must pass through $d + 2k + 2$ points and meet a $2(d-k)$-plane.
\end{proof}

We prove that scrolls satisfy interpolation by induction.
We fix an integer $k$ and induct on the degree of $k$-dimensional varieties.
In order to prove the theorem for a variety of dimension $k$ and degree $d$, we 
make the following inductive hypotheses.
\begin{enumerate}
	\item[\customlabel{custom:ind-high}{ind-1}] When $d > k$, we assume $\minhilb {d-1} k$ satisfies interpolation.
	\item[\customlabel{custom:ind-mid}{ind-2}] When $k < d \leq 2k - 1$, we assume there is a variety in $\minhilb {d-1} k$ containing a general $(2k-d-1)$-plane and containing $2d$ points.
\end{enumerate}

\subsection[Degree at least $2k-1$]{Inductive degeneration for degree at least $2k-1$}

The main aim of this subsection is to prove \autoref{lemma:high-induction},
showing that degree $d$ scrolls satisfy interpolation for $d \geq 2k-1$,
assuming our inductive hypothesis \autoref{custom:ind-high}.

\begin{proposition}
	\label{lemma:high-induction}
	Assuming induction hypothesis ~\ref{custom:ind-high}, if $d > 2k-1$ then $\minhilb d k$ satisfies interpolation.
\end{proposition}

\begin{remark}
	\label{remark:castelnuovos-lemma}
	Before proving \autoref{lemma:high-induction}, let us note that it
	immediately gives another proof that there is a rational normal
	curve through $n + 3$ points in $\bp^n$.
	The proof holds by induction on the degree, which covers all
	cases when $k = 1$ because when $k = 1$, $2k - 1 = 1$.
	This is pictorially summarized in
	\autoref{figure:castelnuovos-lemma}.
\end{remark}

The idea of the proof of \autoref{lemma:high-induction} is to specialize all but two of the points to a hyperplane, and then find a reducible scroll of degree
$d$ which is the union of a scroll of degree $d-1$ in a hyperplane and $\bp^k$, meeting along $\bp^{k-1}$,
as pictured in \autoref{figure:high-induction}.

\begin{figure}
	\centering
	\includegraphics[scale=.35]{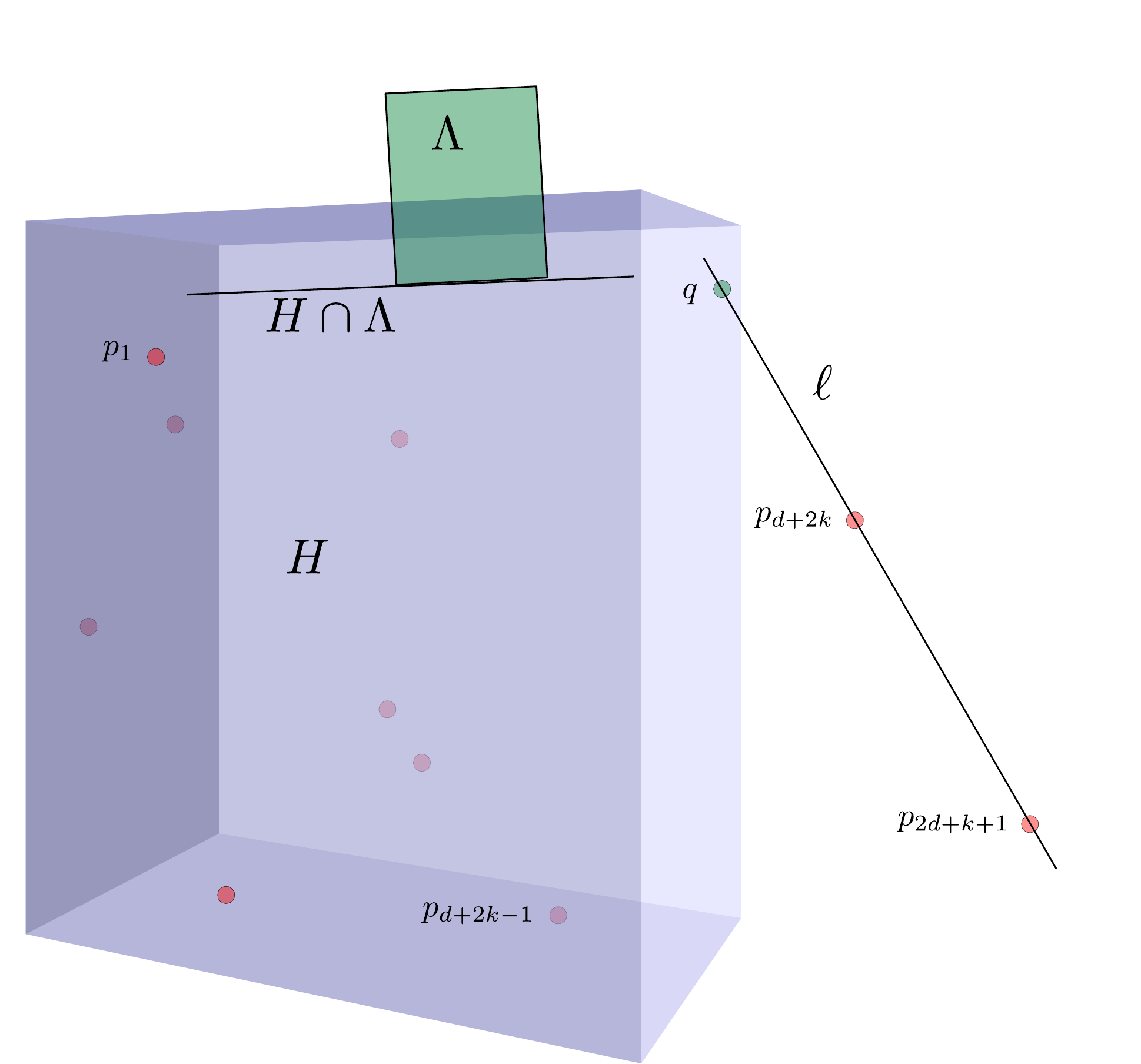}
	\caption{A visualization of the idea of the proof of \autoref{lemma:high-induction},
	where one inductively specializes all but two points to lie in a hyperplane.
}
	\label{figure:high-induction}
\end{figure}
\begin{figure}
	\centering
	\includegraphics[scale=.2]{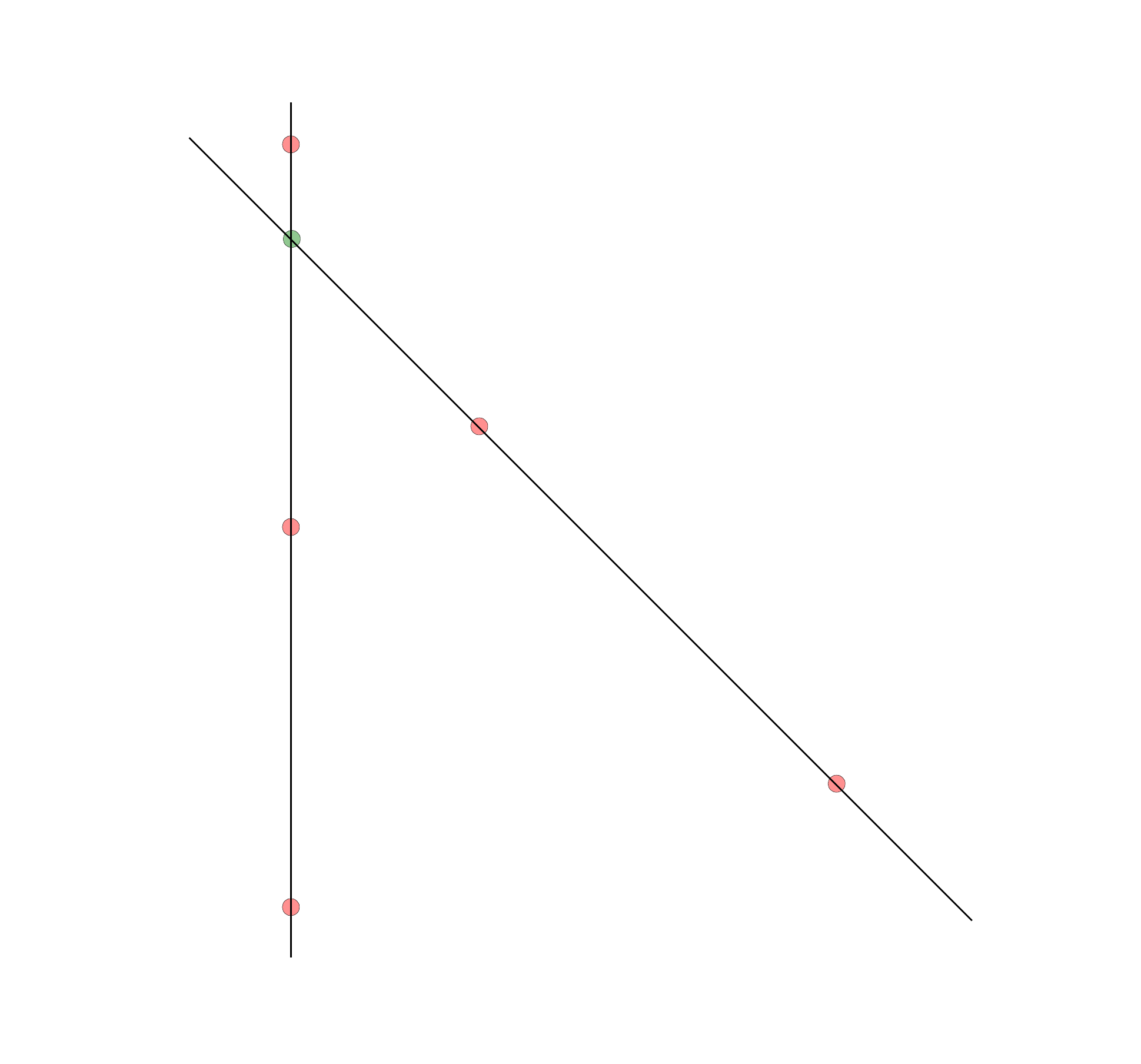}
	\includegraphics[scale=.2]{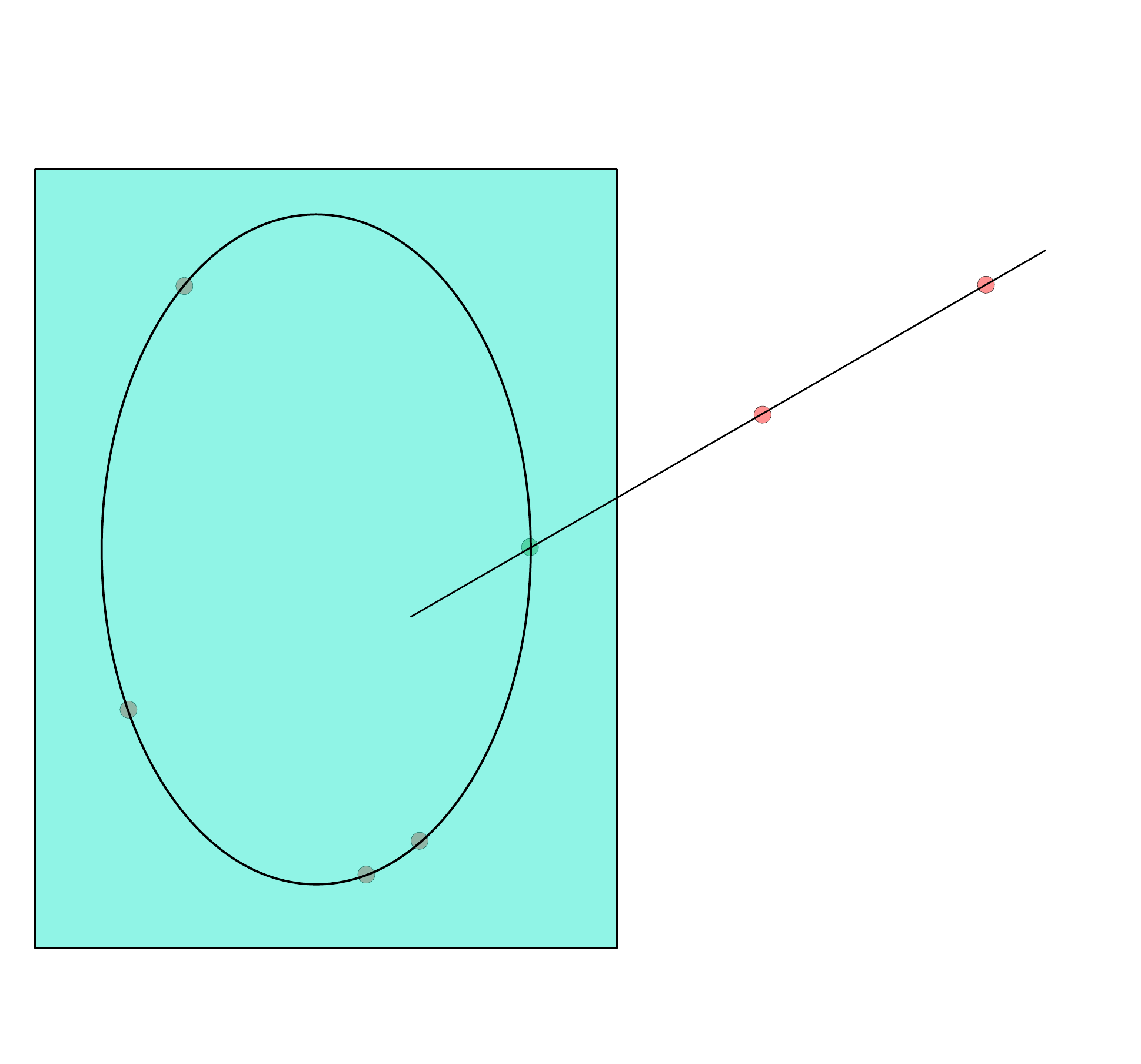}
	\includegraphics[scale=.2]{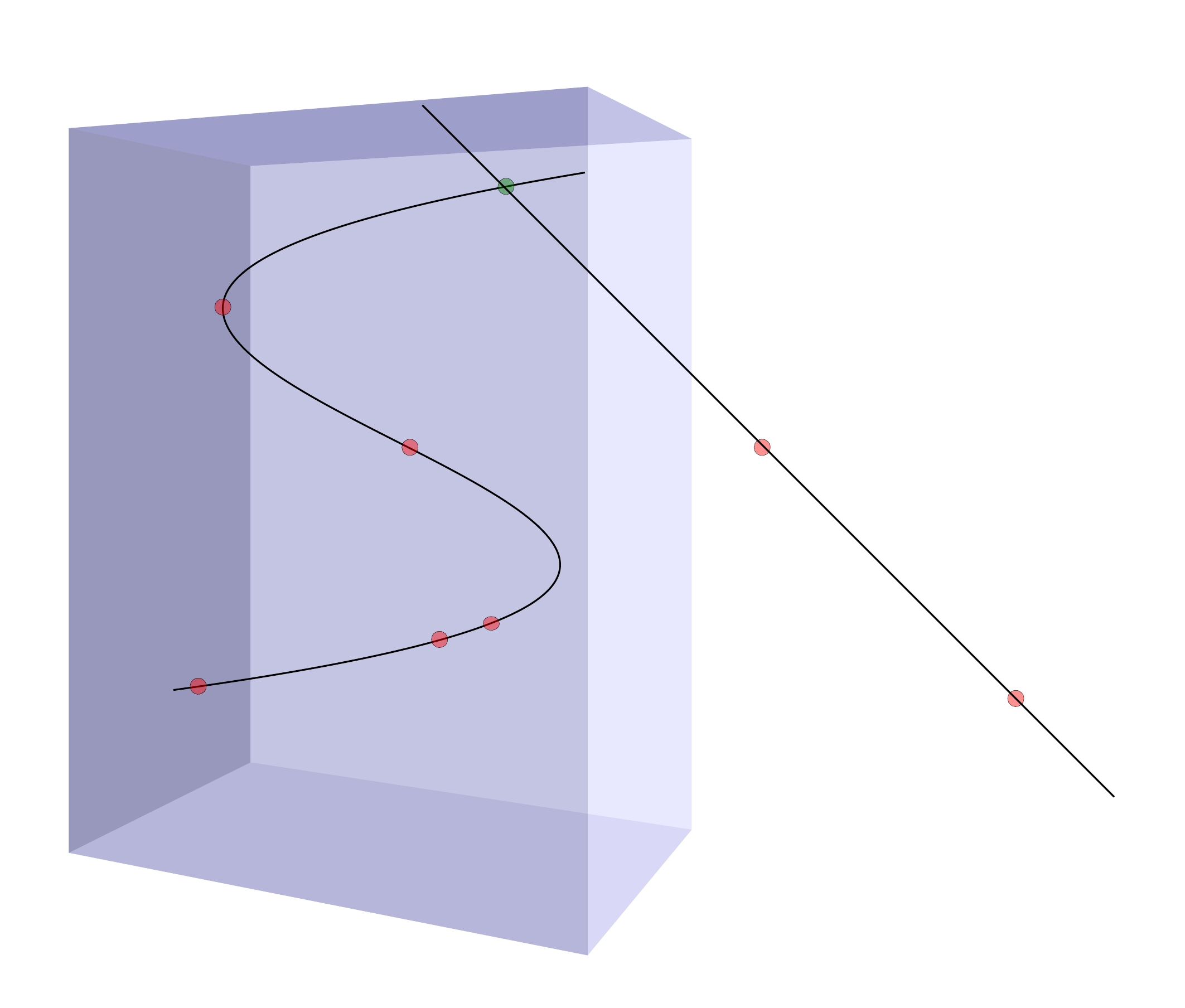}
	\caption{A pictorial description of the inductive degeneration following \autoref{lemma:high-induction} for showing rational normal curves
	satisfy interpolation, see also \autoref{remark:castelnuovos-lemma}. Degenerations are drawn for rational normal curves of degrees 2, 3, and 4.
}
\label{figure:castelnuovos-lemma}
\end{figure}

\begin{proof}
	By \autoref{lemma:interpolation-numerics}, $\minhilb d k$ satisfying interpolation means that given a collection of $d + 2k +1$ points and a general 
	$(d -2k + 1)$-plane, $\Lambda$, we can find 
	some $[Z] \in \minhilb d k$, so that $Z$ passes through
	the points and meets $\Lambda$. Now specialize $d + 2k - 1$ of the points $p_1, \ldots, p_{d+2k-1}$ to a general hyperplane $H \subset \bp^{d+k-1}$. Let $\ell$ be the line through $p_{d+2k}, p_{d+2k+1}$ and let $q := \ell \cap H$.
	By induction hypothesis ~\ref{custom:ind-high}, there is
	a degree $d-1$ dimension $k$ variety of minimal degree which contains the $(d-1) + 2k + 1$ points $p_1, \ldots, p_{d+2k-1}, q$ and meets
	the $((d-1)-2k+1)$-plane $H \cap \Lambda$. Call that variety $X$. Since $q \in X$, there is a unique $(k-1)$-plane
	contained in $X$ and containing $q$, as follows from \autoref{lemma:linear-spaces-in-varieties-of-minimal-degree}.
	Let that $k-1$-plane be $P$ and let $Y$ be the $k$-plane spanned by $\ell$ and $P$.
	Then, the variety $X \cup Y$ with reduced scheme structure lies in $\minhilb d k$ by \autoref{lemma:coskun-degenerations-of-scrolls}.
	
	In order to show $\minhilb d k$ satisfies interpolation, it suffices to show
	$X \cup Y$ is an isolated point in the set of all elements of $\minhilb d k$ containing $p_2, \ldots, p_{d + 2k - 1}$ and $\Lambda$,
by \cite[Thorem A.7(5)]{landesmanP:interpolation-problems:-del-pezzo-surfaces}.
	Since the points and plane were chosen generally subject to the requirement that $k+d$ of the points were contained
	in a hyperplane, by \autoref{lemma:coskun-degenerations-of-scrolls}, it suffices to show there are only finitely many
	scrolls in $\broken d k \cup \minhilbsmooth d k$
	(recall these were defined in \autoref{subsection:notation})  	
	containing $p_1, \ldots, p_{d - 2k+1},$ and meeting
	$\Lambda$.

	This now follows from a dimension count. First, we show there are only finitely many
	scrolls in $\broken d k$ containing $p_1, \ldots, p_{d - 2k+1}$ and meeting $\Lambda$. 
	Because of our specialization of the points, any $[X \cup Y] \in \broken d k$ with $Y \cong \bp^{k-1}$ 
	containing this set of points must satisfy
	$X \subset H$.
	By our induction hypothesis, there are a positive
	finite number of scrolls $X$ in $H$ meeting
	$p_1, \ldots, p_{d-2k-1}, q$, and $\Lambda \cap H$.
	Once $X$ is chosen, $Y$ is uniquely determined, so there
	are finitely many choices for $X \cup Y$.

	To complete the proof, it suffices to show there are only finitely many smooth scrolls containing
	$p_1, \ldots, p_{k+d+1}$ and meeting $\Lambda$.
	This follows from \autoref{lemma:ind-high-smooth-finite}, which we prove next.
\end{proof}

\begin{lemma}
	\label{lemma:ind-high-smooth-finite}
	Let $p_1, \ldots, p_{d+2k+1}$ be points in $\bp^n$ so that $p_1, \ldots, p_{d+2k-1}$ are contained in a hyperplane $H \subset \bp^n$,
	but the points are otherwise general, and let $\Lambda$ be a general $(d-2k+1)$ plane.
	Then, there are only finitely many smooth scrolls containing $p_1, \ldots, p_{d+2k+1}$ and meeting $\Lambda$.
\end{lemma}
\begin{remark}
	\label{remark:}
	Although the following proof is a fairly standard dimension counting argument,
	we include it for completeness.
	Throughout the remainder of the paper, many similar dimension counts will be done, the formal details of which are omitted as they are
	similar in nature to this one.
\end{remark}

\begin{proof}
	For the remainder of the proof, fix a hyperplane $H \subset \bp^n$. Define the incidence correspondence
\begin{align*}
	\Phi := 
	\left( \uhilb X \times_{\hilb {p_1}} H \right) &\times_{\hilb X} \cdots \times_{\hilb X} \left( \uhilb X \times_{\hilb {p_{d+2k-1}}} H \right)
	\\
	&\times_{\hilb X} \uhilb X \times_{\hilb X} \uhilb X \times_{\hilb X} \left( \uhilb X \times_{\bp^n} \hilb \Lambda \right)
\end{align*}
To complete the proof, it suffices to show that 
\begin{align*}
	\dim \Phi &= \dim H^{d+2k-1} \times (\bp^n)^2 \times G(d -2k + 2, n+1).
\end{align*}
because then the projection map
\begin{align*}
	\Phi \ra (\bp^n)^{d+2k-1} \times (\bp^n)^2 \times G(d -2k + 2, n+1)
\end{align*}
is either generically finite or not dominant, and in either case, there
are only finitely many smooth varieties through a general set of such
points, meeting a general plane $\Lambda$.
\begin{align*}
\end{align*}
Note that
\begin{align*}
	\dim \uhilb X \times_{\bp^n} H - \dim \minhilb d k &= k-1 \\
	\dim \uhilb X - \dim \minhilb d k &= k \\
	\dim \uhilb X \times_{\bp^n} \hilb \Lambda - \dim \minhilb d k &= k+ (d-2k+1)(3k-2)
\end{align*}
This completes the proof because
\begin{align*}
	\dim \Phi &= ((d+k)^2 -k^2 - 3) \\
	 & \qquad + \left(  (k-1)(d+2k-1) + (2k) + (k +(d-2k+1)(3k-2))
\right) \\
&= (d+k-2)(d+2k-1)+ 2(d+k-1) \\
& \qquad +(d-2k+2)(n+2k-d-1) \\
&= \dim (\bp^n)^{d+2k-1} \times (\bp^n)^2 \times G(d -2k + 2, n+1).
\end{align*}
\end{proof}
\begin{remark}
	\label{remark:}
	To prove \autoref{lemma:ind-high-smooth-finite}, one can also
	do the following dimension count, although this takes
	a little justification to rigorize:
The condition that $X$ pass through a point imposes $n-k = d-1$ conditions. 
When we specialize a point to a hyperplane, the hyperplane intersects $X$ in a scroll of dimension $k-1$
in $\bp^{n-1}$. Hence, a point still imposes $n-1 - (k-1) = n - k$ conditions. 
Finally, the condition to meet a $(d - 2k + 1)$-plane imposes $n - k - (d-2k+1)$ conditions.
Adding these up, we obtain a total of
\begin{align*}
	(d+2k+1)(n-k) + n-k-(d-2k+1) &= (k+d)^2 - k^2 - 3 \\
	&= \dim \minhilb d k
\end{align*}
conditions, so there are only finitely many elements of
$\minhilb d k$ satisfying these conditions.
\end{remark}

\subsection[Degree between $k+1$ and $2k-1$]{Inductive degeneration for degree between $k+1$ and $2k-1$}

We next show \autoref{corollary:mid-induction},
which lets us verify scrolls of degree between $k+1$ and $2k-1$
satisfy interpolation, assuming \autoref{custom:ind-mid} inductively.

\begin{figure}
	\centering
	\includegraphics[scale=.3]{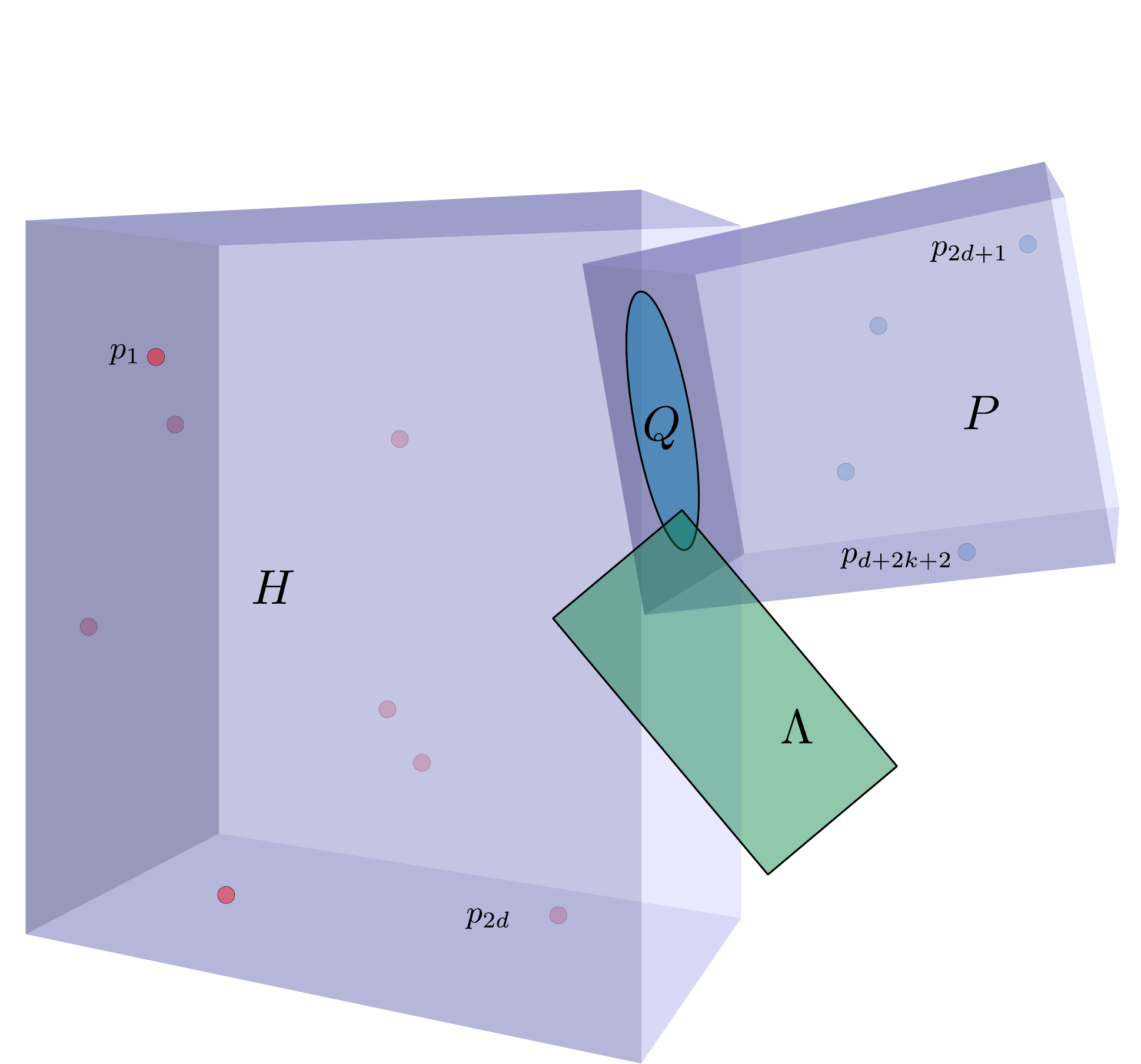}
	\caption{A visualization of the idea of the proof of \autoref{corollary:mid-induction},
	where one inductively specializes $2d$ points to lie in a hyperplane $H$, and the additional plane $\Lambda$
	to meet the intersection of $H$ and the span of the remaining $2k-d+2$ points.
}
	\label{figure:}
\end{figure}

\begin{proposition}
	\label{corollary:mid-induction}
	Let $k+1 \leq d \leq 2k-1$. Assuming induction hypothesis ~\ref{custom:ind-mid} holds for varieties of degree $d-1$ then $\minhilb {d} k$ satisfies interpolation.
\end{proposition}
\begin{proof}
	By \autoref{lemma:interpolation-numerics}, we would like to show there is a degree $d-1$, dimension $k$ variety passing through
	$d + 2k + 2$ points, $p_1, \ldots, p_{d+2k+2}$, and meeting a general $2(d-k)$ plane $\Lambda$. 
	Choose a general hyperplane $H$ and specialize
	$2d$ of these points, $p_{1}, \ldots, p_{2d}$, to $H$.
	Now, the last $2k-d+2$ points remain general, and span some $2k-d+1$ plane $P$.
	Specialize $\Lambda$ so that $\Lambda \cap H \cap P \neq \emptyset$, but so that $\Lambda$ is otherwise general. Let $Q := H \cap P$.

	Then, by induction hypothesis
	~\ref{custom:ind-mid} for varieties of degree $d$, there is a variety of degree $d-1$ containing the $2d$ points
	$p_1, \ldots, p_{2d}$ and containing $Q$. Call this variety $X$. 

	We now explain in detail why the variety $X$ can be taken to be smooth if $p_1, \ldots, p_{2d},Q$
	are chosen generally. Let $W$ be some scroll of degree $d-1$ and dimension $k$, so that $\hilb W = \minhilb {d-1} k$,
	and let $F$ denote the relative Hilbert scheme of $(2k-d)$-planes in $\uhilb W$ over $\hilb W$.
	Define 
	\begin{align*}
	\Phi := F \times_{\hilb W} \uhilb W \times_{\hilb W} \cdots \times_{\hilb W} \uhilb W,
	\end{align*}
	where there are $2d$ copies of $\uhilb W$ in the fiber product.
	Note that the closed points of $\Phi$ correspond to a scroll together with a $(2k-d)$-plane contained in it
	and $2d$ points on it.
	Then, we have natural projection maps
	\begin{equation}
		\nonumber
		\begin{tikzcd}
			\qquad & \Phi \ar {ld}{\pi_1} \ar {rd}{\pi_2} & \\
			\hilb W && G(2k-d+1, n+1) \times \left( \bp^n \right)^{2d}.
		\end{tikzcd}\end{equation}
	The assumption that through a general set of $2d$ points
	and a general $(2k-d)$-plane $Q$ there passes a member of $\hilb W$
	is equivalent to the map $\pi_2$ being dominant.
	Now, since the locus of smooth scrolls $\minhilbsmooth {d-1} k \subset \hilb W$ is dense, 
	it follows that $\pi_1^{-1}\left(\minhilbsmooth {d-1} k \right) \subset \Phi$ is also dense,
	and therefore $\pi_2|_{\pi_1^{-1}\left(\minhilbsmooth {d-1} k\right)}$ is also dominant.
	Dominance is equivalent to the statement that for a general set of $2d$ points and a $2k-d$ plane,
	there is some {\em smooth} scroll passing through the points and containing the plane, as we wanted to show.

	Further, if the points are chosen generally, there is some $(k-1)$-plane $R$ with $Q \subset R \subset X$,
	as we now explain.
	Note, by \autoref{lemma:linear-spaces-in-varieties-of-minimal-degree}, if $\dim Q \neq 1$, then we 
	have that $Q$ is contained in some $(k-1)$-plane. If, instead $\dim Q = 1$, then we know there are at most two components of the Fano scheme of lines
	contained in $X$, and one of these has dimension $k$, which is strictly larger than 
	the other component (if it exists), from \autoref{lemma:linear-spaces-in-varieties-of-minimal-degree}.
	Therefore, since	
	$p_1, \ldots, p_{d+2k+2}$ and $\Lambda$ are chosen generally,
	we may assume that this line $Q$ lies in the component of the Fano scheme of larger dimension. In this case,
	the line $Q$ lies in some $(k-1)$-plane, again by \autoref{lemma:linear-spaces-in-varieties-of-minimal-degree}.

	Define $Y$ to be the plane spanned by $R, p_{2d+1}, \ldots, p_{d+2k+2}$.
	Then, we then obtain that $X \cup Y$ corresponds 
	to a point in $\minhilb d k$ by \autoref{lemma:coskun-degenerations-of-scrolls}.

	Using dimension counts similar to, but more involved than \autoref{lemma:ind-high-smooth-finite}
there are only finitely many such smooth scrolls,
and finitely many such scrolls in $\broken d k$.
By \autoref{lemma:coskun-degenerations-of-scrolls}, $X \cup Y$ 
is an isolated point of the set of schemes containing the points and meeting $\Lambda$, and so it satisfies interpolation by \cite[Theorem A.7]{landesmanP:interpolation-problems:-del-pezzo-surfaces}.
\end{proof}

\subsection{Inductively verifying \getrefnumber{custom:ind-mid}}

In this subsection, we prove \autoref{lemma:middle-induction}, showing \autoref{custom:ind-mid} holds for a given degree
$d$, assuming it holds for degree $d-1$.
The proof of the \autoref{lemma:middle-induction} is quite analogous to that of \autoref{lemma:high-induction}, in that we specialize all but two of the points to
a hyperplane, and then find a scroll of degree d which is a union of a scroll of degree $d-1$ inside a hyperplane and a $\bp^k$, meeting along a $\bp^{k-1}$. We now prove \autoref{lemma:middle-induction} assuming \autoref{proposition:at-most-two-components}.
\begin{figure}
	\centering
	\includegraphics[scale=.3]{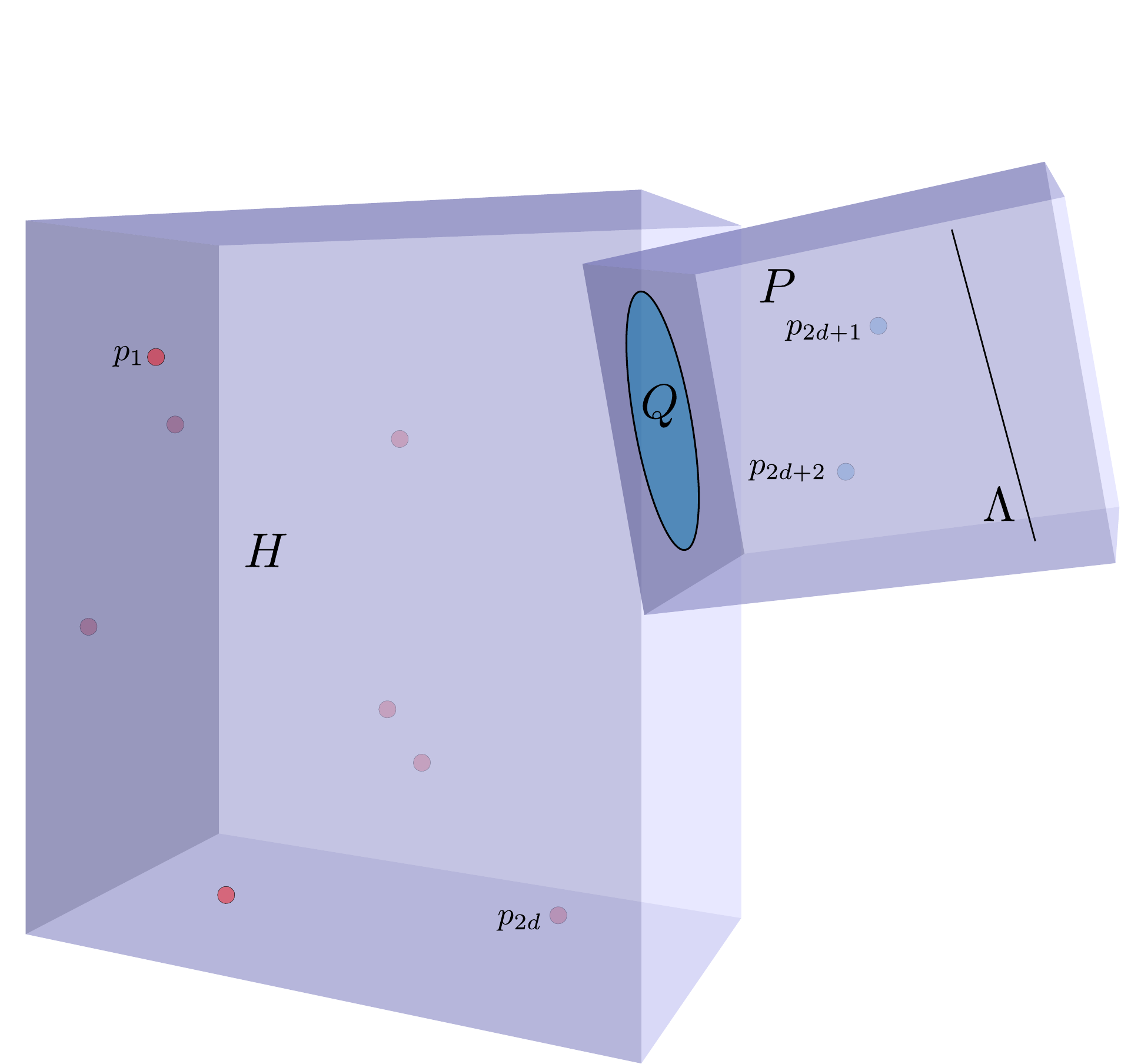}
	\caption{A visualization of the idea of the proof of \autoref{lemma:middle-induction},
	where one inductively specializes $2d$ points to lie in a hyperplane.
}
	\label{figure:}
\end{figure}

\begin{proposition}
	\label{lemma:middle-induction}
	Assuming induction hypothesis \autoref{custom:ind-mid}
	holds for varieties of degree $d-1$ with $k+1 \leq d \leq 2k-2$, 
	it holds for varieties of degree $d$.
\end{proposition}
\begin{proof}[Proof assuming \autoref{proposition:at-most-two-components}]
	We may assume $k \geq 3$ since the case of $k = 1$ is rational normal curves
	and the case $k = 2$ is already covered in ~\cite[Example, p.\ 2]{coskun:degenerations-of-surface-scrolls}.
	To show $\minhilb d k$ satisfies interpolation, we want to show it contains a point corresponding to a variety
	passing through
	$2d+2$ general points $p_1, \ldots, p_{2d+2}$ and a general $(2k-d-1)$-plane $\Lambda$. 

	Choose a general hyperplane $H \subset \bp^{d + k - 1}$ and specialize $\Lambda$ and $p_1, \ldots, p_{2d}$ to be contained in $H$.
	Let $P$ be the $(2k-d+1)$-plane spanned by $\Lambda, p_{2d+1}$ and $p_{2d+2}$, and let $Q := P \cap H$. Then, by inductive hypothesis
	~\ref{custom:ind-mid}, there is a scroll of degree $d - 1$ and dimension $k$, call it $X$,
	containing the $2d$ points $p_1, \ldots, p_{2d}$ and the $2k-d$ plane $Q$.

	There is a unique $(k-1)$-plane $S$ contained in $X$ and containing $Q$, by
	\autoref{lemma:linear-spaces-in-varieties-of-minimal-degree}. 			
	Take $Y$ to be the span of $\Lambda, p_{2d+1}, p_{2d+2},$ and $S$.
	The variety $X \cup Y$ with reduced scheme structure lies in $\broken d k \subset \minhilb d k$
	and contains $p_1, \ldots, p_{2d+2}$, and $\Lambda$,
	as desired.
	Note that the domain and range of the map $\pi_2: \Phi \ra (\bp^n)^{2d+2}\times G(2k-d, n+1)$, as defined in \autoref{proposition:at-most-two-components}, have the same dimension by \autoref{proposition:at-most-two-components}.
	Furthermore, by \autoref{proposition:at-most-two-components},
	when $k > 2$, $\Phi$ has a unique component of maximal dimension,
	call it $\Phi_1$
	and possibly one other component of lower dimension, call it $\Phi_2$, if it exists.
	We may therefore choose
	$(p_1, \ldots, p_{2d+2}, \Lambda)$ generally to avoid
	$\pi_2(\Phi_2)$.

	By a dimension count analogous to \autoref{lemma:ind-high-smooth-finite}, 
	there are only finitely many smooth scrolls and finitely many scrolls in $\broken d k$
	containing $p_1, \ldots, p_{d+2}, \Lambda$.

	So, applying \autoref{lemma:coskun-degenerations-of-scrolls}, to the map $\pi_2|_{\Phi_1}$ we see that 
	the point $[X \cup Y]$ constructed above is isolated among all schemes containing
	$p_1, \ldots, p_{2d+2}, \Lambda$.
	Since $\pi_2$ is a map between proper schemes of the same dimension and has a point isolated
	in its fiber, it is surjective, meaning that 
	\ref{custom:ind-mid} holds for scrolls of degree $d$ and dimension $k$.	
\end{proof}

To complete the proof of \autoref{lemma:middle-induction}, we only need prove \autoref{proposition:at-most-two-components}.

	\begin{lemma}
		\label{proposition:at-most-two-components}
		Let $X \in \minhilb d k$ and let
	$F$ denote the relative Hilbert scheme of $(2k-d-1)$-planes in $\uhilb X$ over $\hilb X$.
	Define
	\begin{align*}
\Phi = \uhilb X \times_{\hilb X} \cdots \times_{\hilb X} \uhilb X \times_{\hilb X} F.
	\end{align*}
	where there are $2d + 2$ copies of $\uhilb X$.
	Define
	\begin{equation}
		\nonumber
		\begin{tikzcd}
			\qquad & \Phi \ar {ld}{\pi_1} \ar {rd}{\pi_2} & \\
			 \minhilb d k  && (\bp^n)^{2d+2} \times G(2k-d, n+1).
		 \end{tikzcd}\end{equation}

	Then, $\Phi$ is irreducible if $2k-d-1 \neq 1$. When $2k -d -1 = 1$, it has at most two components.
		If it has two components, one is of dimension
		\begin{align*}
			\dim (\bp^n)^{2d+2} \times G(2k-d, n+1) &= (k+d-1)(2d+2) + (2k-d)(2d-k)
		\end{align*}
		and the other is of dimension 
		\begin{align*}
			(k+d-1)(2d+2) + (2k-d)(2d-k) - 2(k-2).
		\end{align*}
	\end{lemma}
	\begin{proof}
		First, the statements regarding irreducibility of $\Phi$ follow immediately from the statements on irreducibility
		of the Fano scheme as given in \autoref{lemma:linear-spaces-in-varieties-of-minimal-degree},
		coupled with the assumption that $\hilb X$ and $\uhilb X$ are irreducible.

		Since \autoref{lemma:linear-spaces-in-varieties-of-minimal-degree}, also gives the dimensions
		of the irreducible components of $F$, call them $F_i$, we have
		that the dimension of the $i$th irreducible component of $\Phi$ is
		$(\dim F_i - \dim \hilb X) + (2d+2)(\dim \uhilb X - \dim \hilb X) + \dim \hilb X$, as claimed.
		\end{proof}

\subsection{Base case: verifying \getrefnumber{custom:ind-mid} for degree $k$ varieties}
\label{subsection:base-case-plane}

We have essentially completed our inductive steps, as pictured in
\autoref{figure:induction-schematic}. 
The final step is to examine the base case of interpolation for
degree $k$ and dimension $k$ varieties.
Recall that smooth degree $k$, dimension $k$ varieties are realized as
a Segre embedding $\bp^1 \times \bp^{k-1} \ra \bp^{2k-1}$, by \autoref{lemma:segre-isomorphic-to-scroll}.

We have two remaining parts of our argument for showing scrolls satisfy interpolation.
First, in this subsection, we show that we can find a variety of degree $k$ and dimension $k$ containing
$2k+2$ general points and a general $(k-1)$-plane.

Then, in \autoref{subsection:base-case-points},
we show that $\minhilb k k$ satisfies interpolation, meaning that we can find
a scroll in $\minhilb k k$ through $3k+3$ general points.

\begin{figure}
	\centering
\includegraphics[scale=.3]{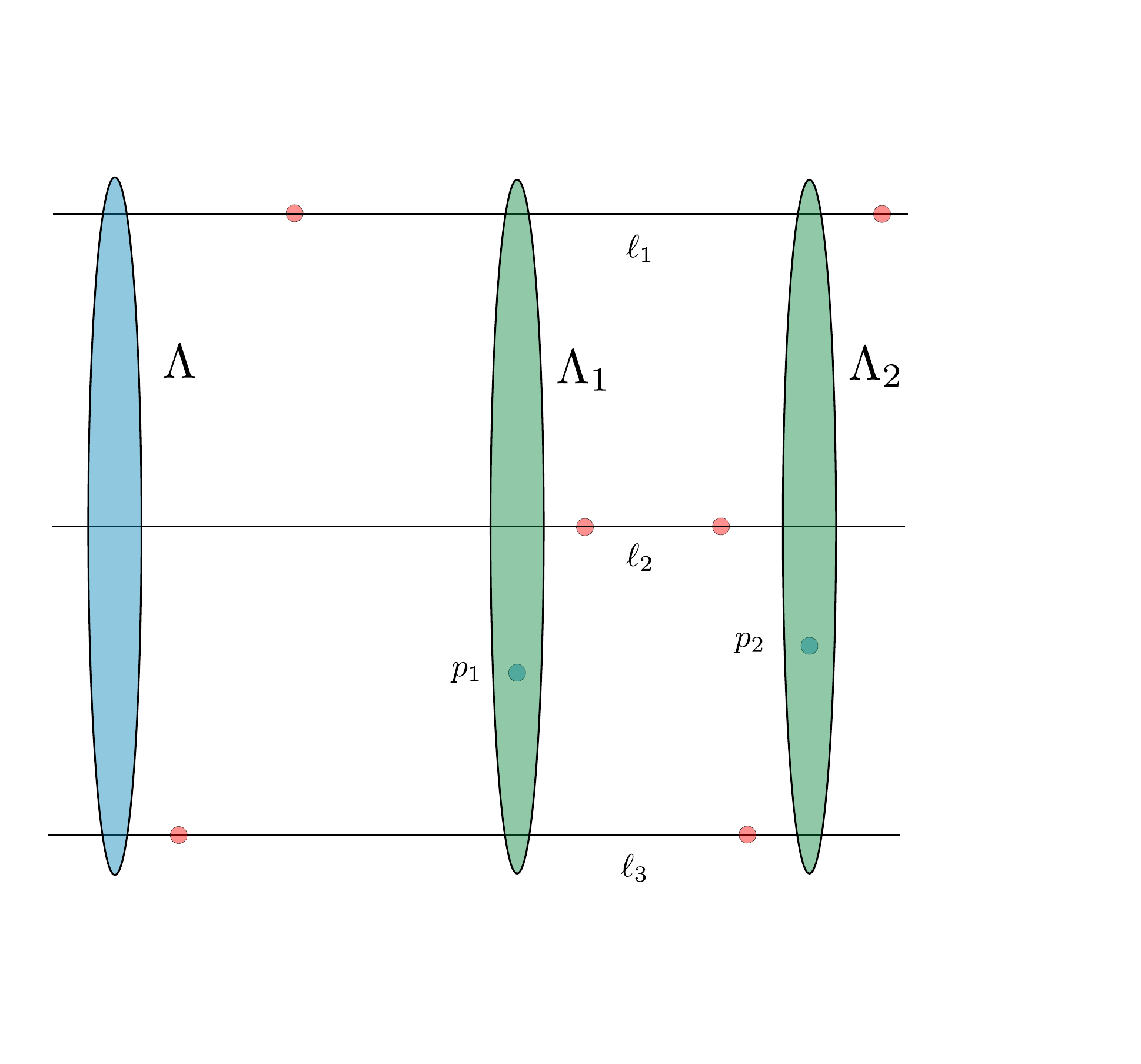}
\caption{A visualization of the idea of the proof of \autoref{proposition:segre-plane},
where one specializes pairs of points to lie on lines meeting the given $(k-1)$-plane $\Lambda$.
}
\label{figure:}
\end{figure}

\begin{proposition}
	\label{proposition:segre-plane}
	There exist finitely many varieties of minimal degree $\scroll {1^k}$
	containing $2k + 2$ general points and a general $(k-1)$-plane.
\end{proposition}

\begin{proof}
	The case $k \leq 2$ is easy so we assume $k \geq 3$.
	Label the points by $p_1, \ldots, p_{2k+2}$ and the $(k-1)$-plane by $\Lambda$.
	Define the $k$ lines
	$\ell_1, \ldots, \ell_k$, by $\ell_i := \overline {p_{2i+1}, p_{2i+2}}$.
	Specialize the points $p_1, \ldots, p_{2k}$ to general points
	satisfying the condition that the $k$ lines
	$\ell_1, \ldots, \ell_k$ all meet
	$\Lambda$. We claim that there are only finitely many varieties $\scroll {1^k}$ which
	contain such a configuration of points and $\Lambda$.

	To see this, note that the ideal of a scroll is generated by quadrics,
	by \cite[p.\ 6]{eisenbudH:on-varieties-of-minimal-degree}. Since any such $\scroll {1^k}$ 
	contains
	three points on each line $\ell_i$, namely $p_{2i+1}, p_{2i+2}$ and $\ell_i \cap \Lambda$, the line $\ell_i \subset X$.

	Hence, it suffices to show there are a finite, nonzero number of smooth
	varieties $\scroll {1^k}$ corresponding to a point
	in $\minhilb k k$ that
	\begin{enumerate}
		\item contain a $(k-1)$-plane $\Lambda$
		\item contain $k$ lines $\ell_1, \ldots, \ell_k$, and
		\item contain $2$ points $p_1, p_2$.
	\end{enumerate}

		Let $X$ be some smooth scroll containing $\Lambda, \ell_1, \ldots, \ell_k, p_1, p_2$.
	By \autoref{lemma:linear-spaces-in-varieties-of-minimal-degree}, since we are assuming $k \geq 3$, the only $(k-1)$-planes contained in $X$ are the fibers of the projection to $\bp^1$.
Since $\ell_1, \ldots, \ell_k$ all meet $\Lambda$ at precisely 1 point, they must be lines
of the type \ref{custom:line-small}. Additionally, the two remaining points $p_1$ and $p_2$ must be contained in two $(k-1)$-planes $\Lambda_1, \Lambda_2$ with
	$\Lambda_i \cap \ell_j \neq \emptyset$.
		By Schubert calculus, there are a positive finite number of planes containing $q_1$ which meet $\ell_1, \ldots, \ell_k$,
	(In fact, the number of such planes is equal to the number of standard Young
	tableaux of size $k-1$, although for our argument, we only need that this number is nonzero.) Therefore, there are 
	a positive finite number of choices of pairs of planes $\Lambda_1, \Lambda_2$
	containing $p_1, p_2$, and meeting lines $\ell_1, \ldots, \ell_k$.

	Now, for any such choice of planes $\Lambda_1, \Lambda_2$, there is a unique smooth scroll containing
	$\ell_1, \ldots, \ell_k, \Lambda_1, \Lambda_2, \Lambda$ by \autoref{lemma:unique-segre-containing-linear-spaces}.
	Hence, in total, there are
	a finite number of smooth scrolls containing $\Lambda, \ell_1, \ldots, \ell_k, p_1, p_2$, as claimed.
\end{proof}

\subsection{Base case: interpolation for degree $k$ varieties}
\label{subsection:base-case-points}

Before piecing our inductive argument for interpolation of scrolls together,
we have to show that $\minhilb k k$ satisfies interpolation, which we now do.
We do this by constructing an isolated point in a
fiber of the map $\pi_2$ from \eqref{equation:projections-segre-incidence-correspondence-with-lines-and-points}.
This fiber has positive dimensional
components, and so it takes
some care to show that this is indeed an isolated point.

\begin{proposition}
	\label{proposition:segre-lines-interpolation}
	Let $k \geq 2$ and 
	$p_1, \ldots, p_{3k+3}$ be $3k + 3$ general points
	in $\bp^{2k-1}$.
	Then, there is some variety of minimal degree $\scroll {1^k}$ containing $p_1, \ldots, p_{3k+3}$.
	That is, $\minhilb k k$ satisfies interpolation.
\end{proposition}
\begin{figure}
	\centering
	\includegraphics[scale=.35]{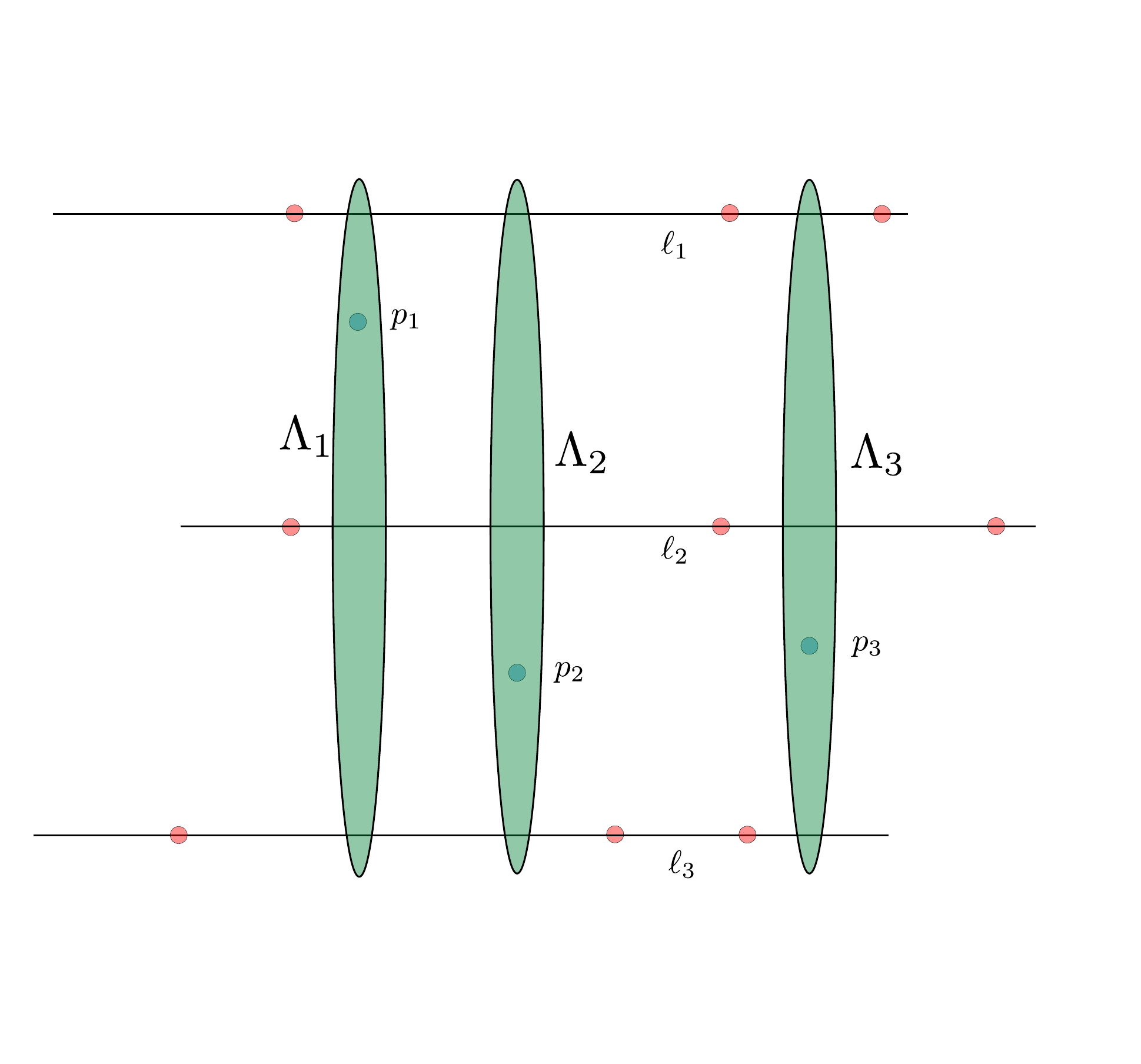}
	\caption{A visualization of the idea of the proof of \autoref{proposition:segre-lines-interpolation},
	where one specializes triples of points to lie on lines, and then finds a scroll containing those points.
}
	\label{figure:isolated-segre-threefold}
\end{figure}

\begin{proof}
	To begin with, specialize the points, as follows:
	Pick $k$ lines $\ell_{1}, \ldots, \ell_k$ so that $\ell_1, \ldots, \ell_k$
	span $\bp^{2k-1}$.
	Now, specialize the points so that
	$p_{3i+1}, p_{3i+2}, p_{3i+3}$ are distinct points on $\ell_{i}$ for $1 \leq i \leq k$.
	This leaves us with $3$ general points, which we have labeled as $p_1, p_2, p_3$. 
	
	Let $X$ be a smooth scroll of type $(1^k)$ and let $\sch^{t+1}(\uhilb X/\hilb X)$ denote the relative Hilbert 
	scheme of lines in $\uhilb X$ over $\hilb X$.
	Define the incidence correspondences
	\begin{align}
		\nonumber
		\Psi &:= \uhilb X \times_{\hilb X} \cdots \times_{\hilb X} \uhilb X, \\
\label{equation:correspondence-with-lines-and-points}
\Phi &:= \sch^{t+1}(\uhilb X/\hilb X) \times_{\hilb X} \cdots \times_{\hilb X} \sch^{t+1}(\uhilb X/\hilb X)
			\times_{\hilb X} \uhilb X\times_{\hilb X} \uhilb X\times_{\hilb X} \uhilb X,
	\end{align}
		where $\Psi$ has a product taken over $3k+3$ copies of $\uhilb X$ and $\Phi$ has a product taken over $k$ copies of $\sch^{t+1}(\uhilb X/\hilb X)$. Set theoretically (but without
	specifying a scheme structure)
	we can describe the closed points of $\Psi$ as
	\begin{align*}
	\label{equation:segre-incidence-correspondence-with-points}
		\left\{ \left(X, r_1, \ldots, r_{3k+3}\right) \subset \minhilb k k  \times (\bp^{2k-1})^{3k+3} : r_i \in X \right\}
	\end{align*}
	and the closed points of $\Phi$ as
	\begin{align}
			\nonumber
		\left\{ \left( L_1, \ldots, L_k, r_1, r_2, r_3, X \right) \subset (G(2,2k))^k \times (\bp^{2k-1})^3 \times \minhilb k k : L_i \subset X, r_i \in X  \right\}.
	\end{align}
	We have projections
	\begin{equation}
		\label{equation:projections-segre-incidence-correspondence-with-lines-and-points}
		\begin{tikzcd} 
		\qquad & \Psi \ar{dl}{\eta_1} \ar {dr}{\eta_2}& \\
		\minhilb k k && (\bp^{2k-1})^{3k+3}. \\
			\qquad & \Phi \ar{dl}{\pi_1} \ar {dr}{\pi_2}& \\
			\minhilb k k && (G(2, 2k))^k \times (\bp^{2k-1})^3.
		\end{tikzcd}\end{equation}

	Define $\scx = \eta_2^{-1} (p_1, \ldots, p_{3k+3})$
	and
	$\scy := \pi_2^{-1}(\ell_1, \ldots, \ell_k, p_1, p_2, p_3)$.
	To complete the proof, it suffices to show
	$\scx$ has an isolated point.
	Note that the locus of $\scx$ corresponding
	to smooth scrolls is the same as the locus
	of $\scy$ corresponding
	to smooth scrolls, since any scroll
	containing $p_4, \ldots, p_{3k+3}$ contains the lines
	$\ell_1, \ldots, \ell_k$ as scrolls are generated
	by quadrics.
	Therefore, to show $\scx$ has an isolated point,
	it suffices to show $\scy$ has an isolated point
	corresponding to a smooth scroll.

	Since we choose $\ell_1, \ldots, \ell_k, p_1, p_2, p_3$
	generally, by Schubert calculus, there are a positive finite number of $(k-1)$-planes containing $p_i$ and meeting all lines
	$\ell_1, \ldots, \ell_k$.
	In particular, there is a positive, finite number of three tuples of $(k-1)$-planes $(\Lambda_1, \Lambda_2, \Lambda_3)$ so that $\Lambda_i$ contains
	$p_i$ and meets $\ell_1, \ldots, \ell_k$. Further, since our parameters
	were chosen generally, 
	by \autoref{lemma:unique-segre-containing-linear-spaces}, there is a unique
	Segre variety $W$ containing any given $\Lambda_1, \Lambda_2,\Lambda_3$ and $\ell_1, \ldots, \ell_k$.

	This corresponds to a point $[W]$ in $\scy$, hence also some $[W']$ in $\scx$. 
	To complete the proof, it remains to show that $[W']$ is an isolated point in $\scy$.
	This follows from \autoref{lemma:isolated-segre}. 	
\end{proof}

\begin{lemma}
	\label{lemma:isolated-segre}
	Consider the incidence correspondence
	$\Phi$ defined in
	\eqref{equation:correspondence-with-lines-and-points}
	with projection maps $\pi_1$ to $\minhilb k k$ and $\pi_2$ to
	$(G(2, 2k))^k \times (\bp^{2k-1})^3$, as defined in
	\eqref{equation:projections-segre-incidence-correspondence-with-lines-and-points}.
	Let $\ell_1, \ldots, \ell_k$ be lines in $\bp^{2k-1}$ and let $p_1, p_2,p_3 \in \bp^{2k-1}$ be three points.
		Define $\scy := \pi_2^{-1}(\ell_1, \ldots, \ell_k, p_1, p_2, p_3)$.
		Write $[W'] \in \scy$ as
		$W' \cong \bp^1 \times \bp^{k-1}$ via
		\autoref{lemma:segre-isomorphic-to-scroll}. 
		If the lines $\ell_i$ are all 
		of type ~\ref{custom:line-small} 
		Then, $[W']$ is an isolated point of $\scy$.
		\end{lemma}
\begin{proof}
	We assume $k \geq 3$, as the case $k = 2$ holds since $S_{1,1}$ is a quadric hypersurface.

	First, by Schubert calculus and \autoref{lemma:unique-segre-containing-linear-spaces},
	there are only finitely many points  $[Y] \in \scy$ so that all $\ell_i$ appear as lines of type
	~\ref{custom:line-small}.
	To conclude the proof, it suffices to show that there cannot be a family of scrolls in $\scy$
	containing $W'$ so that in the general member of the family, there is some $\ell_i$ which appears as a line
	of type ~\ref{custom:line-large}.
	So, suppose for the sake of contradiction, that such a family $f: \scz \rightarrow C$, for $C\subset \scy$ and $[W'] \in C$, does exist.
We may assume that $C$ is reduced and $1$ dimensional (by replacing $C$ with
$1$ dimensional reduced closed subscheme) and that $C$ is smooth and connected
(by replacing $C$ with a connected component of its normalization).

	Then, consider the relative Hilbert scheme $\sch := \sch^{t+1}(\scz/C)$ of lines inside $\scz$ over $C$. 
	We claim that $\sch$ has two connected components: one whose lines are of type
	~\ref{custom:line-small} and the other whose lines are of type ~\ref{custom:line-large}.
	In particular, this suffices to show such a family over $C$ cannot exist,
	as such a family would intersect two distinct connected
	components of $\sch$.

	To see that $\sch$ has the two connected components, we examine the map
	$f$. Since the fiber of $\scz$ over a closed point of $C$ is a scroll,
	we know that the fibers of $g:\sch \ra C$ over any closed point of $C$ has two smooth connected components
	of different dimensions, using
	\autoref{lemma:linear-spaces-in-varieties-of-minimal-degree},
	and our standing assumption that $k > 2$.
	
	Now, by Stein factorization, we can factor
	\begin{equation}
		\nonumber
		\label{equation:}
		\begin{tikzcd}
			\sch \ar {rr}{\alpha} \ar {rd}{g} && D \ar {ld}{\beta} \\
			 & C & 
		 \end{tikzcd}\end{equation}
	 where $\alpha$ has connected fibers, $\beta$ is finite,
	 and $D$ is a normal curve, hence smooth.
	 Further, since $g$ is a dominant map to a smooth
	 curve, we obtain that
	 $\alpha$ is also a dominant map to a smooth curve, hence flat.
	 But, since the connected components
	 of the closed fibers of $g$ have different dimensions,
	 the connected components of the closed fibers of $\alpha$ have different dimensions,
	 hence different Hilbert polynomials.
	 Since $\alpha$ is flat, $D$ must be disconnected, and hence $\sch$ must be disconnected.
\end{proof}

\subsection{Spelling out the induction}
In this subsection, we combine the previous results from this section to prove that smooth scrolls satisfy interpolation.
We then conclude that all varieties of minimal degree satisfy interpolation, and hence also strong interpolation in characteristic
$0$. See \autoref{figure:induction-schematic} for a visualization 
of the proof of \autoref{theorem:scrolls-interpolation}.

\begin{figure}
	\centering
\begin{equation}
  \nonumber
\begin{tikzpicture}[baseline= (a).base]
\node[scale=.45] (a) at (0,0){
  \begin{tikzcd}[column sep=tiny]
    \qquad && && && && && & (d = 2k+1) \ar{dll}{\begin{NoHyper}\autoref{lemma:high-induction}\end{NoHyper}} \\ 
    \qquad && && && && & (d = 2k) \ar{dll}{\begin{NoHyper}\autoref{lemma:high-induction}\end{NoHyper}} && \\
    \qquad && && && & (d = 2k-1) \ar{dlll}{\begin{NoHyper}\autoref{corollary:mid-induction}\end{NoHyper}} &&  && \\
    \qquad && && \left[d = 2k-2\right] \ar{dl}[swap]{\begin{NoHyper}\autoref{lemma:middle-induction}\end{NoHyper}} && (d=2k-2) \ar{dlll}{\begin{NoHyper}\autoref{corollary:mid-induction}\end{NoHyper}}&  &&  && \\
    \qquad && & \left[d = 2k - 3\right] \ar{dl}[swap]{\begin{NoHyper}\autoref{lemma:middle-induction}\end{NoHyper}} && (d=2k-3) \ar{dlll}{\begin{NoHyper}\autoref{corollary:mid-induction}\end{NoHyper}} &&  &&  && \\
    \qquad && \iddots \ar{dl}[swap]{\begin{NoHyper}\autoref{lemma:middle-induction}\end{NoHyper}}&& \iddots \ar{dlll}{\begin{NoHyper}\autoref{corollary:mid-induction}\end{NoHyper}} &  &&  &&  && \\
    \qquad & \left[d = k+1\right] \ar{dl}[swap]{\begin{NoHyper}\autoref{lemma:middle-induction}\end{NoHyper}} && (d = k + 1) \ar{dlll}{\begin{NoHyper}\autoref{corollary:mid-induction}\end{NoHyper}} && &&  &&  && \\
    \substack{\begin{NoHyper}\autoref{proposition:segre-plane}\end{NoHyper}}\hspace{.2cm}\left[d = k\right]  && \substack{\begin{NoHyper}\autoref{proposition:segre-lines-interpolation}\end{NoHyper}}\hspace{.2cm}(d = k)  & && &&  &&  &&
  \end{tikzcd}
};
\end{tikzpicture}
\end{equation}
\caption{
This is a schematic diagram for the proof that scrolls satisfy interpolation.
The parenthesized expressions $(d = a)$ indicate that
scrolls of degree $a$ satisfy interpolation,
while the bracketed expressions $[d = a]$ indicate that scrolls of degree $a$
satisfy the hypothesis \ref{custom:ind-mid}.
See \autoref{theorem:scrolls-interpolation} for a written proof.
The arrows point from higher degree to lower degree, and are labeled
by the proposition showing that interpolation for the lower degree
variety implies interpolation for the higher degree variety.
}
	\label{figure:induction-schematic}
\end{figure}
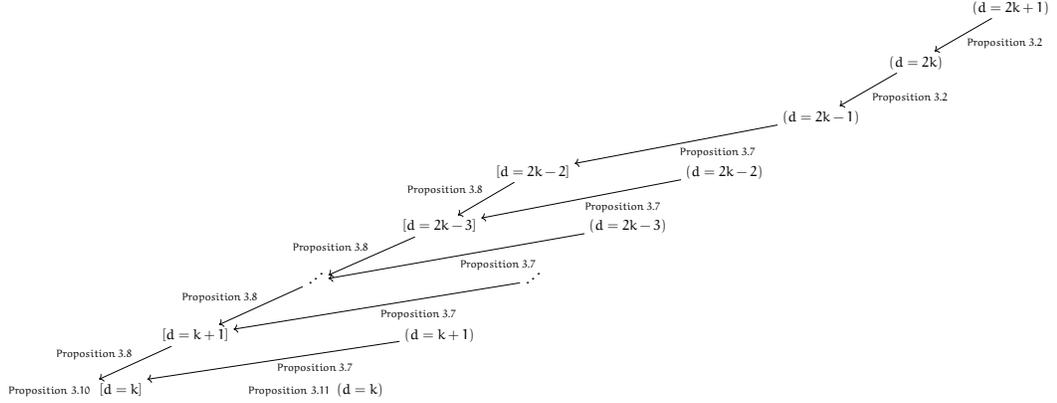

\begin{theorem}
	\label{theorem:scrolls-interpolation}
	If $X$ is a rational normal scroll of minimal degree
	then $\hilb X$ satisfies interpolation.
\end{theorem}
\begin{proof}
	First, if $X = \scroll {1^k}$ then this holds by
	\autoref{proposition:segre-lines-interpolation} with
	$j = 0$. 

	Second, suppose $X = \scroll {2^t, 1^{k-t}}$ with
	$1 \leq t \leq k-1$. We next show $X$ satisfies interpolation. By \autoref{proposition:segre-plane}, inductive hypothesis ~\ref{custom:ind-mid}
	holds for $d = k+1$. By induction, assume it holds for a given degree $d-1$.
	Then, By \autoref{lemma:middle-induction}, it holds for degree $d$, provided
	$d \leq 2k-2$. Finally, we obtain that varieties of degree $t + k$
	satisfy interpolation by \autoref{corollary:mid-induction}.

	To complete the proof, it suffices to show that $X$ satisfies interpolation
	when $d > 2k - 1$. We have shown this when $d = 2k -1$. Inductively assume that
	inductive hypothesis ~\ref{custom:ind-high} holds for degree $d-1$. Then,
	by \autoref{lemma:high-induction}, it also holds for degree
	$d$. Therefore, varieties of degree $d$ satisfy interpolation.
\end{proof}

We now prove our main theorem.

\minimalInterpolation*
\begin{proof}
	First, by \cite[Theorem 1]{eisenbudH:on-varieties-of-minimal-degree},
	we only need show that quadric surfaces, 
	scrolls, and the Veronese embedding $\bp^2 \ra \bp^5$ satisfy
	interpolation.

	First, quadric surfaces satisfy interpolation because all hypersurfaces do
	and the 2-Veronese surfaces satisfies interpolation by
	\cite[Theorem 5.6]{landesmanP:interpolation-problems:-del-pezzo-surfaces}.
	Finally, by \autoref{theorem:scrolls-interpolation}
	if $X$ is a scroll then $\hilb X$ satisfies interpolation.
\end{proof}

In order to state an immediate corollary of
\autoref{theorem:interpolation-minimal-surfaces}, recall the notion of strong interpolation given in 
\cite[Definition A.5]{landesmanP:interpolation-problems:-del-pezzo-surfaces}: An irreducible subscheme of the Hilbert scheme $\sch$
parameterizing varieties of dimension $k$ in $\bp^n$
satisfies {\bf strong interpolation}
if for any collection of planes $\Lambda_1, \ldots, \Lambda_m$
of dimensions $\lambda_1 \geq \cdots \geq \lambda_m$ with
$0 \leq \lambda_i \leq n - k$ and $\sum_{i=1}^m \lambda_i \leq \dim \sch$,
there is some $[Y] \in \sch$ meeting all of $\Lambda_1, \ldots, \Lambda_m$. We say an integral variety $X$ lying on a unique irreducible
component of the Hilbert scheme satisfies strong interpolation
if $\hilb X$ does.

\begin{corollary}
	\label{corollary:strong-interpolation-for-varieties-of-minimal-degree}
	If $\bk$ has characteristic $0$, then smooth varieties of minimal degree
	over $\bk$ satisfy strong interpolation.
\end{corollary}
\begin{proof}
	This follows from \autoref{theorem:interpolation-minimal-surfaces}
	and the equivalence of interpolation and strong interpolation
	in characteristic $0$, as follows from the equivalence of conditions
	$(1)$ and $(v)$ from
	\cite[Theorem A.7]{landesmanP:interpolation-problems:-del-pezzo-surfaces}, using \autoref{proposition:smooth-minimal-degree-hilbert-scheme}.
\end{proof}

\begin{remark}[Strong interpolation in positive characteristic]
	\label{remark:}
	While \autoref{corollary:strong-interpolation-for-varieties-of-minimal-degree}
	shows that varieties of minimal degree satisfy strong interpolation in characteristic
	$0$, we note that 
	the question of whether varieties of minimal degree satisfy strong interpolation in positive
characteristic is still an open problem. 

One particularly tantalizing example is that of the $2$-Veronese
surface in characteristic $2$.
It is shown in \cite[Corollary 7.2.9]{landesman:undergraduate-thesis}
that the normal bundle of the $2$-Veronese surface does not satisfy interpolation
in characteristic $2$. (See \cite[Definition A.5]{landesmanP:interpolation-problems:-del-pezzo-surfaces}
for a definition of normal bundle interpolation.)
However, it is unknown whether
the $2$-Veronese surface satisfies strong interpolation in characteristic $2$.
\end{remark}

\bibliographystyle{alpha}
\bibliography{/home/aaron/Dropbox/master}

\end{document}